\documentclass[a4paper, 10pt, american]{amsart}

\usepackage[colorlinks,pdfpagelabels,pdfstartview = FitH,bookmarksopen
= true,bookmarksnumbered = true,linkcolor = blue,plainpages =
false,hypertexnames = false,citecolor = red,pagebackref=false]{hyperref}
\usepackage{times,latexsym,amssymb}
\usepackage{amsmath,amsthm,bm}
\usepackage{graphics, epsfig}
\usepackage{color}
\usepackage{appendix}
\usepackage[margin=1.5in]{geometry}

\usepackage{amsbsy}
\usepackage{amstext}
\usepackage{amssymb}
\usepackage{esint}
\usepackage{stmaryrd}
\let\TeXchi\chi
\newbox\chibox
\setbox0 \hbox{\mathsurround0pt $\TeXchi$}
\setbox\chibox \hbox{\raise\dp0 \box 0 }
\def\chi{\copy\chibox}

\newcommand{\dx}{\mathrm{d}x}

\newcommand{\dt}{\mathrm{d}t}
\newcommand{\ds}{\mathrm{d}s}

\renewcommand{\epsilon}{\varepsilon}
\renewcommand{\rho}{\varrho}

\author[N. Liao]{Naian Liao}
\address{Naian Liao\\
Fachbereich Mathematik, Universit\"at Salzburg\\
Hellbrunner Str. 34, 5020 Salz\-burg, Austria}
\email{naian.liao@sbg.ac.at}

\author[L. Sch\"atzler]{Leah Sch\"atzler}
\address{Leah Sch\"atzler\\
Fachbereich Mathematik, Universit\"at Salzburg\\
Hellbrunner Str. 34, 5020 Salzburg, Austria}
\email{leahanna.schaetzler@sbg.ac.at}



\begin{document}
%% if using the Springer package svmno.cls the 
%% theorem envronment is already defined
%% if not remove the comment below
%%%%%%%%%% LEMMA, THM, PROP, SECTION %%%%%%%%%%%%%%%
%\newtheorem{theorem}{Theorem}[section]
\newtheorem{proposition}{Proposition}[section]
\newtheorem{theorem}{Theorem}[section]
\newtheorem{lemma}{Lemma}[section]
\newtheorem{corollary}{Corollary}[section]
\newtheorem{remark}{Remark}[section]
%\newtheorem{proof}{Proof}
%%%%%%%%%%%%%%%%%%%%%%%%%%%%%%%%%%%%%%%%%%%%%%%%%%%%%%%%%%
\renewcommand{\thesection}{\arabic{section}}
\renewcommand{\theequation}{\thesection.\arabic{equation}}
\renewcommand{\thetheorem}{\thesection.\arabic{theorem}}
\numberwithin{equation}{section}
\numberwithin{theorem}{section}
\numberwithin{proposition}{section}
\numberwithin{lemma}{section}
\numberwithin{remark}{section}
\setcounter{secnumdepth}{3}
%%%%%%%%%%%%%%%%%%%%%%%%%%%%%%%%%%%%%%%%%%%%%%%%%%%%%%%%%%%%%%%
%%%%%%%%%%%%%%%%%%%%%%%%%%%%%%%%%%%%%%%%%%%%%%%%%%%%%%%%%%%%%%%
%%%%%%%%%%%%%%%%%%%%%%%%%%%%%%%%%%%%%%%%%%%%%%%%%%%%%%%%%%%%%%%
%%%% FORMATTING MACROS %%%%%%
\newcommand{\cl}{\centerline}
\newcommand{\sms}{\smallskip}
\newcommand{\ms}{\medskip}
\newcommand{\bs}{\bigskip}
\newcommand{\noi}{\noindent}
\newcommand{\itl}[1]{\textit{#1}}
\newcommand{\blf}[1]{\textbf{#1}}
\newcommand{\dsty}{\displaystyle}
\newcommand{\txty}{\textstyle}
\newcommand{\ssty}{\scriptstyle}
\newcommand{\tty}{\texttt}

%%%%%%%%%% SPECIAL SYMBOLS %%%%%%%%%%%%

\newcommand\Par{\mathhexbox278\,}

%%%% GREEK LETTERS  %%%%%%%%

\newcommand{\al}{\alpha}
\newcommand{\Al}{\Alpha}
\newcommand{\be}{\beta}
\newcommand{\Be}{\Beta}
\newcommand{\Gm}{\Gamma}
\newcommand{\gm}{\gamma}
\newcommand{\dl}{\delta}
\newcommand{\Dl}{\Delta}
\newcommand{\lm}{\lambda}
\newcommand{\Lm}{\Lambda}
\newcommand{\kp}{\kappa}
\newcommand{\varep}{\varepsilon}
\newcommand{\eps}{\epsilon}
\newcommand{\vp}{\varphi}
\newcommand{\sig}{\sigma}
\newcommand{\Sig}{\Sigma}
\newcommand{\om}{\omega}
\newcommand{\Om}{\Omega}
\newcommand{\uom}{\mbox{\boldmath$\omega$}}
\newcommand{\btau}{\mbox{\boldmath$\tau$}}
\newcommand{\bnu}{\mbox{\boldmath$\nu$}}
\newcommand{\up}{\upsilon}
\newcommand{\z}{\zeta}

%%% SPECIAL MATH SYMBOLS %%%%%%%%%%%%

\newcommand{\df}[1]{\buildrel\mbox{\small def}\over{#1}}
\newcommand{\op}[1]{\buildrel\mbox{\tiny o}\over{#1}}
\newcommand{\db}{\prime\prime}
\newcommand{\bsl}{\backslash}
\newcommand{\lb}{\lbrack\!\lbrack}
\newcommand{\rb}{\rbrack\!\rbrack}
\newcommand\la{\langle}
\newcommand\ra{\rangle}
\newcommand{\ev}{\equiv}
\newcommand{\nev}{\not\equiv}
\newcommand{\nn}{\mathbb{N}}
\newcommand{\qq}{\mathbb{Q}}
\newcommand{\zz}{\mathbb{Z}}
\newcommand{\rr}{\mathbb{R}}
\newcommand{\rn}{\rr^N}
\newcommand{\cc}{\mathbb{C}}
\newcommand{\id}{\mathbb{I}}
\newcommand{\bo}{\mathbb{O}}

\newcommand{\amsb}[1]{\mathbb{#1}}
\newcommand{\mcl}[1]{\mathcal{#1}}
\newcommand{\bl}[1]{\mathbf{#1}}
\newcommand{\ov}[1]{\overline{#1}}
\newcommand{\wt}[1]{\widetilde{#1}}
\newcommand{\wh}[1]{\widehat{#1}}

\newcommand{\llra}{\leftrightarrow}
\newcommand{\lra}{\longrightarrow}
\newcommand{\LLR}{\Longleftrightarrow}
\newcommand{\LRA}{\Longrightarrow}
\newcommand{\LLA}{\Longleftarrow}

%%%% BLACK BOX AND OPEN BOX %%%%%%%%%%%%%%%%

\newcommand{\bbox}{\vrule height.6em width.6em 
depth0em} %%%%% Black Box
\newcommand{\os}{\vbox{\hrule \hbox{\vrule 
height.6em depth0pt 
\hskip.6em \vrule height.6em depth0em}
\hrule}} %%%%%% Open Square

%%%%%%%%%%%%%% OPERATORS %%%%%%%%%%%%%%%%%%%

%\newcommand{\ker}{\operatorname{ker}}
\newcommand{\dvg}{\operatorname{div}}
\newcommand{\curl}{\operatorname{curl}}
\newcommand{\supp}{\operatorname{supp}}
\newcommand{\essup}{\operatornamewithlimits{ess\,sup}}
\newcommand{\essinf}{\operatornamewithlimits{ess\,inf}}
\newcommand{\essosc}{\operatornamewithlimits{ess\,osc}}
\newcommand{\osc}{\operatornamewithlimits{osc}}
\newcommand{\sign}{\operatorname{sign}}
\newcommand{\loc}{\operatorname{loc}}
\newcommand{\diam}{\operatorname{diam}}
\newcommand{\dist}{\operatorname{dist}}
\newcommand{\card}{\operatorname{card}}
\newcommand{\meas}{\operatorname{meas}}
\newcommand{\spn}{\operatorname{span}}
\newcommand{\dtm}{\operatorname{det}}
%

%%%%%%%% OVER AND UNDER LIMITS %%%%%%%%%%%

\newcommand{\overlim}{\mathop{\overline{\lim}}\limits}
\newcommand{\underlim}{\mathop{\underline{\lim}}\limits}
\newcommand{\ttop}[2]{\genfrac{}{}{0pt}{}{#1}{#2}}
\newcommand{\bcu}{\mathop{\txty{\bigcup}}\limits}
\newcommand{\bca}{\mathop{\txty{\bigcap}}\limits}
\newcommand{\bsu}{\mathop{\txty{\sum}}\limits}
\newcommand{\pro}{\mathop{\txty{\prod}}\limits}

%%%%%%%%%%%  DERIVATIVES %%%%%%%%%%%%%%%%

\newcommand{\pl}{\partial}
\newcommand{\ptt}{\frac{\pl}{\pl t}}
\newcommand{\ppx}{\frac\pl{\pl x}}
\newcommand{\dds}{\frac d{ds}}
\newcommand{\ddt}{\frac d{dt}}

%%%%%%%%%%%%%%%%%%%%%%%%%%%%%%%%%%%%%%%%%%%%%%%%
%%%%%%%%%%%%%  INTEGRALS  %%%%%%%%%%%%%%%
%%%%%%%%%%%%%%%%%%%%%%%%%%%%%%%%%%%%%%%%%%%%%%%%
\newcommand{\intl}{\int\limits}
\newcommand{\iintl}{\iint\limits}
%%%%%%%%%%%%%%%%%%%%%%%%%%%%%%%%%%%%%%%%%%%%%%%%
%%%%%%%%%%%%%  INTEGRAL AVERAGES  %%%%%%%%%%%%%%%
%%%%%%%%%%%%%%%%%%%%%%%%%%%%%%%%%%%%%%%%%%%%%%%%
\def\Xint#1{\mathchoice
    {\XXint\displaystyle\textstyle{#1}}%
    {\XXint\textstyle\scriptstyle{#1}}%
    {\XXint\scriptstyle\scriptscriptstyle{#1}}%
    {\XXint\scriptscriptstyle\scriptscriptstyle{#1}}%
    \!\int}
\def\XXint#1#2#3{\setbox0=\hbox{$#1{#2#3}{\int}$}
    \vcenter{\hbox{$#2#3$}}\kern-0.5\wd0}
\def\bint{\Xint-}
\def\dashint{\Xint{\raise4pt\hbox to7pt{\hrulefill}}}
\def\dashiint{\bint\kern-0.15cm\bint}
% integral averages

\newcommand{\ovl}[3]{\int_{#1}^{#2}\kern-#3pt\raise4pt\hbox to7pt{\hrulefill}\ }

\newcommand{\ovll}[3]{\intl_{#1}^{#2}\kern-#3pt\raise4pt\hbox to7pt{\hrulefill}\ }

\newcommand{\tvl}[2]{\iint_{#1}\kern-#2pt\raise4pt\hbox to7pt{\hrulefill}\ }

%%%%%%%%%%% ANALYSIS MACROS %%%%%%%%%%%%%%%%%%%

%%%% Domain \Om %%

\newcommand{\omt}{\Om_T}
\newcommand{\plo}{\partial\Omega}
\newcommand{\ovo}{\bar{\Om} }

%% C-infinity Local spaces and symbols
%
\newcommand{\ci}[1]{C^\infty\!\left({#1}\right)}
\newcommand{\cio}[1]{C_o^\infty\!\left({#1}\right)}
\newcommand{\lloc}[1]{L_{\loc}\!\left({#1}\right)}
\newcommand{\xy}{|x-y|}

%% Integrals

\newcommand{\intom}{\intl_{\Om}}
\newcommand{\intbo}{\intl_{\plo}}
\newcommand{\inom}{\int_{\Om}}
\newcommand{\inbo}{\int_{\plo}}
\newcommand{\intrn}{\intl_{\rn}}

%%%%%%%%%%%% ENDING THE DOCUMENT %%%%%%%%%%%%%%%%

\newcommand{\bye}{\end{document}}

%%%%%%%%%%%%%%%%%%%%%%%%%%%%%%%%%%%%%%%%%%%%%%%%%

%\hsize=5in
%\vsize=7.5in

%
%% Tag
%
%\def\uptag#1#2$${\iftagsleft@\leqno\else\eqno\fi
%  \hbox{\def\pagebreak{\global\postdisplaypenalty-\@M}%
%  \def\nopagebreak{\global\postdisplaypenalty\@M}\rm(#1\unskip)${}^{#2}$}%
%  $$\postdisplaypenalty\z@\ignorespaces}
%

%\def\btag#1#2$${\iftagsleft@\leqno\else\eqno\fi
%  \hbox{\def\pagebreak{\global\postdisplaypenalty-\@M}%
%  \def\nopagebreak{\global\postdisplaypenalty\@M}\rm(#1\unskip)${}_{#2}$}%
%  $$\postdisplaypenalty\z@\ignorespaces}
%
%\def\ubtag#1#2#3$${\iftagsleft@\leqno\else\eqno\fi
%  \hbox{\def\pagebreak{\global\postdisplaypenalty-\@M}%
%  \def\nopagebreak{\global\postdisplaypenalty\@M}
%  \rm(#1\unskip)${}_{#2}^{#3}$}%
%  $$\postdisplaypenalty\z@\ignorespaces}
%
%\def\ptag#1$${\iftagsleft@\leqno\else\eqno\fi
%  \hbox{\def\pagebreak{\global\postdisplaypenalty-\@M}%
%  \def\nopagebreak{\global\postdisplaypenalty\@M}
%  \rm(#1\unskip)${}^{\prime}$}%
%  $$\postdisplaypenalty\z@\ignorespaces}
%
%% tag wih no brackets
%
%\def\etag#1$${\iftagsleft@\leqno\else\eqno\fi
%  \hbox{\def\pagebreak{\global\postdisplaypenalty-\@M}%
%  \def\nopagebreak{\global\postdisplaypenalty\@M}\rm#1\unskip}%
%  $$\postdisplaypenalty\z@\ignorespaces}
%

%%%%%% DEFINIZIONE DI DATA GIORNO, MESE, ANNO %%%%%%%%%%
%
%\newcommand{\data}{\number\day\space \ifcase\month\or
%Gennaio\or Febbraio\or Marzo\or Aprile\or Maggio\or 
%Giugno\or
%Luglio\or Agosto\or Settembre\or Ottobre\or Novembre\or
%Dicembre\fi,\space\number\year}
%

%\input harnack_mono.mac
%%%%%%%%%%%%%%%%%%%%%%%%%%%%%%%%%%%%%%%%%%%%%%%
\title[H\"older regularity for a doubly nonlinear equation]{On the H\"older regularity of signed solutions to a doubly nonlinear equation. Part III}

\date{}
\begin{abstract}
We establish the local H\"older continuity of possibly sign-changing solutions to
a class of doubly nonlinear parabolic equations whose prototype is
\[
\partial_t\big(|u|^{q-1}u\big)-\Dl_p u=0,\quad 1<p<2,\quad 0<p-1<q.
\]
The proof exploits the space expansion of positivity for the singular, parabolic $p$-Laplacian
and employs the method of intrinsic scaling by carefully balancing the double singularity.
%The arguments are flexible enough to obtain the boundary regularity for
 %initial-boundary value problems of Dirichlet type and Neumann type.
\vskip.2truecm
%\noindent{\bf AMS Subject Classification (2020):} 
\end{abstract}
\maketitle

%{\small\tableofcontents}
%\tableofcontents

%%%%%%%%%%%%%%%%%%%%%%%%%%%%%%%%%%%%
\section{Introduction and Main Results}
This is the third paper in a project devoted to integrating
the fragmented theory of H\"older regularity of weak solutions to
a class of doubly nonlinear parabolic equations, whose model case is
\begin{equation*}%\label{prototype}
	\partial_t\big(|u|^{q-1}u\big)- \dvg (|Du|^{p-2}Du) =0\quad\mbox{ weakly in $ E_T$.}
\end{equation*}
Here $E_T:=E\times(0,T]$ for some $T>0$  and some $E$ open in $\rn$. 
%$\Dl_p:=\dvg (|Du|^{p-2}Du)$ is the $p$-Laplace operator.
Indeed, we have studied the borderline case, i.e., $p>1$ and $q=p-1$ in \cite{BDL}; 
whereas the doubly degenerate case $p>2$ and $0<q<p-1$ has been treated in \cite{BDLS-1}.
In this note, we will push the front line forward by taking on the doubly singular case $1<p<2$ and $0<p-1<q$.
 
Our approach to the H\"older regularity is entirely local and structural.
In fact, we shall consider  parabolic partial differential equations of the general form
\begin{equation}  \label{Eq:1:1}
	\partial_t\big(|u|^{q-1}u\big)-\dvg\bl{A}(x,t,u, Du) = 0\quad \mbox{ weakly in $ E_T$}
\end{equation}
where the function $\bl{A}(x,t,u,\z)\colon E_T\times\rr^{N+1}\to\rn$ is only assumed to be
measurable with respect to $(x, t) \in E_T$ for all $(u,\z)\in \rr\times\rn$,
continuous with respect to $(u,\z)$ for a.e.~$(x,t)\in E_T$,
and subject to the structure conditions
\begin{equation}\label{Eq:1:2p}
	\left\{
	\begin{array}{c}
		\bl{A}(x,t,u,\z)\cdot \z\ge C_o|\z|^p \\[5pt]
		|\bl{A}(x,t,u,\z)|\le C_1|\z|^{p-1}%
	\end{array}
	\right .
	\qquad \mbox{for a.e.~$(x,t)\in E_T$, $\forall\,u\in\rr$, $\forall\,\z\in\rn$,}
\end{equation}
where $C_o$ and $C_1$ are given positive constants.

In the sequel, we will refer to the set of parameters
$\{p,\,q,\,N,\,C_o,\,C_1\}$ as  the structural {\it data}.
We also write $\boldsymbol \gm$ as a generic positive constant that can be quantitatively
determined a priori only in terms of the data and that can change from line to line.

Let $u\in L^{\infty}(E_T)$ and set $M:=\|u\|_{\infty, E_T}$. Define $\Gm:=\pl E_T-\overline{E}\times\{T\}$
to be the parabolic boundary of $E_T$, and for a compact set $\mathcal{K}\subset E_T$
introduce the intrinsic, parabolic distance from $\mathcal{K}$ to $\Gm$ by
\begin{equation*}
	\begin{aligned}
		\dist(\mathcal{K};\,\Gm)&:=\inf_{\substack{(x,t)\in \mathcal{K}\\(y,s)\in\Gm}}
		\left\{|x-y|+M^{\frac{p-q-1}p}|t-s|^{\frac1p}\right\}.
	\end{aligned}
\end{equation*}

Postponing the formal definitions of weak solution,
now we state the main result concerning the interior H\"older continuity of weak solutions
to \eqref{Eq:1:1}, subject to the structure conditions \eqref{Eq:1:2p} for $1<p<2$ and $0<p-1<q$.
\begin{theorem}\label{Thm:1:1}
	Let $u$ be a bounded, local, weak solution to \eqref{Eq:1:1} -- \eqref{Eq:1:2p} in $E_T$.
	Then $u$ is locally H\"older continuous in $E_T$. More precisely,
	there exist constants $\boldsymbol\gm>1$ and $\be\in(0,1)$ that can be determined a priori only in terms of the data, such that for every compact set $\mathcal{K}\subset E_T$,
	\begin{equation*}
	\big|u(x_1,t_1)-u(x_2,t_2)\big|
	\le
	\boldsymbol \gm M
	\left(\frac{|x_1-x_2|+M^{\frac{p-q-1}p}|t_1-t_2|^{\frac1p}}{\dist(\mathcal{K};\Gm)}\right)^{\be},
	\end{equation*}
for every pair of points $(x_1,t_1), (x_2,t_2)\in \mathcal{K}$.
\end{theorem}
%%%%%
\begin{remark}\label{Rmk:1:1}\upshape
Local boundedness of weak solutions suffices in Theorem~\ref{Thm:1:1}.
Moreover, our arguments also apply to equations with lower order terms.
Modifications can be modeled on the arguments in \cite[Chapters~II -- IV]{DB}
and in \cite[Appendix~C]{DBGV-mono}.
Nevertheless, we will not pursue such generality here but concentrate on the actual novelty.
\end{remark}
\begin{remark}\upshape
Theorem~\ref{Thm:1:1} yields a Liouville type theorem by an argument similar to \cite[Corollary~1.1]{BDL}, 
which we refer to for details.
\end{remark}

\subsection{Novelty and Significance}\label{S:NS}
The issue of H\"older regularity for the doubly singular equation \eqref{Eq:1:1} with the general structure \eqref{Eq:1:2p}
in the case $1<p<2$ and $0<p-1<q$ (see Figure~\ref{Fig:range}),
has been considered by a number of authors, cf.~\cite{ChenZhou95, Henriques-20, Ivanov-2, Vespri-Vestberg}. 
However, the theory is fragmented
in the sense that either they further restrict the range of $p$ and $q$, or they consider non-negative solutions only.
Moreover, different notions of solution could have been used.

\begin{figure}[h]\label{Fig:range}
\centering
\includegraphics[scale=0.8]{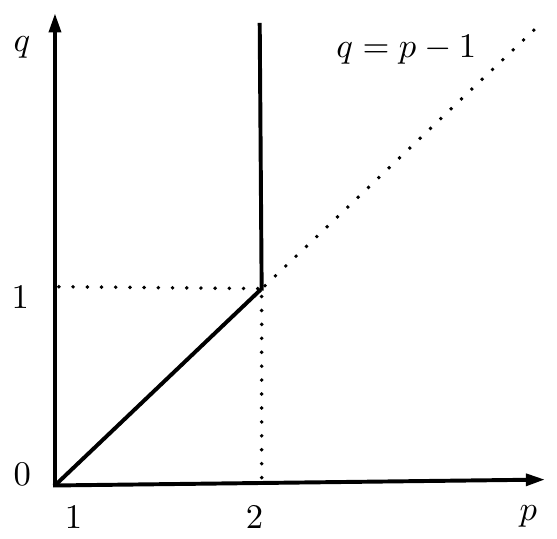}
\caption{Range of $p$ and $q$ }
\end{figure}

Our contribution to the existing literature lies in that a unified approach to the H\"older regularity is provided 
in the doubly singular range, that is, $1<p<2$ and $0<p-1<q$,
and sign-changing solutions are dealt with.

The key observation is that once given the space expansion of positivity for
the singular, parabolic $p$-Laplacian, the proof can be achieved with other components 
that rely solely on the energy estimate for the doubly singular equation.
Even though we follow in spirit the intrinsic scaling argument of \cite{ChenDB92} in the context of singular, parabolic $p$-Laplacian,
more careful analysis has to be carried out in order to balance the double singularity exhibited by \eqref{Eq:1:1}.
Therefore, a challenge here is the delicate, technical nature of the intrinsic scaling argument.
Furthermore, it is worth pointing out that, like in our previous works \cite{BDL, BDLS-1}, 
we dispense with any kind of logarithmic type energy estimate, cf.~\cite[p.~28]{DB}.

Our argument can also be adapted to the borderline cases. In particular,
when $q=1$ and $1<p<2$, it recovers the proof in \cite{ChenDB92};
when $q>1$ and $p=2$, it deals with the fast diffusion equation;
when $q=p-1$, we return to our first work \cite{BDL} (see also \cite{Trud2} for non-negative solutions).

The interior arguments do not apply to the boundary regularity directly except at the initial level,
cf.~\cite[Theorem~1.2, Remark~1.3]{BDLS-1}.
We leave this issue to a future investigation.

%%%%%%%%%%%%%%%%%
\subsection{Notion of Solution}\label{S:1:2}
A function
\begin{equation}  \label{Eq:1:3p}
	u\in C\big(0,T;L^{q+1}_{\loc}(E)\big)\cap L^p_{\loc}\big(0,T; W^{1,p}_{\loc}(E)\big)
\end{equation}
is a local, weak sub(super)-solution to \eqref{Eq:1:1} with the structure
conditions \eqref{Eq:1:2p}, if for every compact set $K\subset E$ and every sub-interval
$[t_1,t_2]\subset (0,T]$
\begin{equation}  \label{Eq:1:4p}
	\int_K |u|^{q-1}u\z \,\dx\bigg|_{t_1}^{t_2}
	+
	\iint_{K\times (t_1,t_2)} \big[-|u|^{q-1}u\z_t+\bl{A}(x,t,u,Du)\cdot D\z\big]\dx\dt
	\le(\ge)0
\end{equation}
for all non-negative test functions
\begin{equation*}
\z\in W^{1,q+1}_{\loc}\big(0,T;L^{q+1}(K)\big)\cap L^p_{\loc}\big(0,T;W_o^{1,p}(K)%
\big).
\end{equation*}
This guarantees that all the integrals in \eqref{Eq:1:4p} are convergent.

A function $u$ that is both a local weak sub-solution and a local weak super-solution
to \eqref{Eq:1:1} -- \eqref{Eq:1:2p} is a local weak solution.

%%%%%%%%%%%%%%%%%
\medskip
{\it Acknowledgement.} 
N.~Liao has been supported by the FWF-Project P31956-N32
``Doubly nonlinear evolution equations".

%%%%%%%%%%%%%%%%%%%%%%%%%

\section{Energy Estimates}
In this section we deduce certain energy estimates for weak sub(super)-solutions.
The different roles played by sub-solutions and super-solutions are emphasized.

We first introduce some notations. 
 For any $k\in\rr$, let
\[
(u-k)_+=\max\{u-k,0\}, \qquad (u-k)_-=\max\{-(u-k),0\}.
\]

For $k, w\in\rr$ we define two quantities
\begin{equation*}
	\mathfrak g_\pm (w,k)=\pm q\int_{k}^{w}|s|^{q-1}(s-k)_\pm\,\ds.
\end{equation*}
Note that $\mathfrak g_\pm (w,k)\ge 0$.
For $b\in\rr$ and $\al>0$, we will embolden $\boldsymbol{b}^\al$ to denote the
signed $\al$-power of $b$ as 
\begin{align*}%\label{Eq:power}
\boldsymbol{b}^\al=
\left\{
\begin{array}{cl}
|b|^{\al-1}b, &b\neq0,\\[5pt]
0, &b=0.
\end{array}\right.
\end{align*}
Throughout the rest of this note, 
we will use the symbols 
\begin{equation*}
\left\{
\begin{aligned}
(x_o,t_o)+Q_\rho(\theta)&:=K_{\rho}(x_o)\times(t_o-\theta\rho^p,t_o),\\[5pt]
(x_o,t_o)+Q(R,S)&:=K_R(x_o)\times (t_o-S,t_o),
\end{aligned}\right.
\end{equation*} 
to denote (backward) cylinders with the indicated positive parameters;
when the context is unambiguous, we will omit the vertex $(x_o,t_o)$ from the symbols for simplicity.

The following energy estimate for local weak sub(super)-solutions defined in Section~\ref{S:1:2}
has been presented in \cite{BDLS-1}, which actually holds for all $p>1$ and $q>0$.

\begin{proposition}\label{Prop:2:1}
	Let $u$ be a  local weak sub(super)-solution to \eqref{Eq:1:1} -- \eqref{Eq:1:2p} in $E_T$.
	There exists a constant $\boldsymbol \gm (C_o,C_1,p)>0$, such that
 	for all cylinders $Q(R,S)=K_R(x_o)\times (t_o-S,t_o)\Subset E_T$,
 	every $k\in\rr$, and every non-negative, piecewise smooth cutoff function
 	$\z$ vanishing on $\pl K_{R}(x_o)\times (t_o-S,t_o)$,  there holds
\begin{align*}
	\max \bigg\{ \essup_{t_o-S<t<t_o}&\int_{K_R(x_o)\times\{t\}}	
	\z^p\mathfrak g_\pm (u,k)\,\dx,
	\iint_{Q(R,S)}\z^p|D(u-k)_\pm|^p\,\dx\dt \bigg\}\\
	&\le
	\boldsymbol \gm\iint_{Q(R,S)}
		\Big[
		(u-k)^{p}_\pm|D\z|^p + \mathfrak g_\pm (u,k)|\partial_t\z^p|
		\Big]
		\,\dx\dt\\
	&\phantom{\le\,}
	+\int_{K_R(x_o)\times \{t_o-S\}} \z^p \mathfrak g_\pm (u,k)\,\dx.
\end{align*}
\end{proposition}
Here and in the sequel, when we state {\it ``$u$ is a sub(super)-solution...''}
and use $``\pm"$ or $``\mp"$ in what follows, we mean the sub-solution corresponds to
the upper sign and the super-solution corresponds to the lower sign in the statement.
%%%%%%%%%%%%%%%%%%%%%%%%%

\section{Preliminary Tools} 
For a compact set $K\subset\rr^N$ and
 a cylinder $\mathcal{Q}:=K\times(T_1,T_2]\subset E_T$
we introduce numbers $\boldsymbol\mu^{\pm}$ and $\boldsymbol\om$ satisfying
\begin{equation*}
	\boldsymbol\mu^+\ge\essup_{\mathcal{Q}}u,
	\quad 
	\boldsymbol\mu^-\le\essinf_{\mathcal{Q}} u,
	\quad
	\boldsymbol\om\ge\boldsymbol\mu^+-\boldsymbol\mu^-.
\end{equation*}
In this section, we collect some lemmas, which will be the main ingredients
in the proof of Theorem~\ref{Thm:1:1}. 
The following De Giorgi type lemma holds true for all $p>1$ and $q>0$ and has been stated in \cite[Lemma~3.1]{BDLS-1}.
For the detailed proof of the statement in a slightly different formulation, we refer to \cite[Lemma~2.2]{Liao-21}.

\begin{lemma}\label{Lm:DG:1}
 Let $u$ be a locally bounded, local weak sub(super)-solution to \eqref{Eq:1:1} -- \eqref{Eq:1:2p} in $E_T$.
 Set $\theta=(\xi\boldsymbol\om)^{q+1-p}$ for some $\xi\in(0,1)$ and assume $(x_o,t_o) + Q_\varrho(\theta) \subset \mathcal{Q}$.
There exists a constant $\nu\in(0,1)$ depending only on 
 the data, such that if
\begin{equation*}
	\Big|\Big\{
	\pm\big(\boldsymbol \mu^{\pm}-u\big)\le \xi\boldsymbol\om\Big\}\cap (x_o,t_o)+Q_{\varrho}(\theta)\Big|
	\le
	\nu|Q_{\varrho}(\theta)|,
\end{equation*}
then either
\begin{equation*}
	|\boldsymbol\mu^{\pm}|>8\xi\boldsymbol\om,
\end{equation*}
or
\begin{equation*}
	\pm\big(\boldsymbol\mu^{\pm}-u\big)\ge\tfrac{1}2\xi\boldsymbol\om
	\quad
	\mbox{a.e.~in $(x_o,t_o)+Q_{\frac{1}2\varrho}(\theta)$.}
\end{equation*}
\end{lemma}
%%%%%%%%%
The next lemma propagates measure information in time.
An analogous result has been presented in \cite[Lemma~4.1]{BDL} for $q=p-1$.
It is noteworthy that this lemma actually holds true for all $p>1$ and $q>0$.
\begin{lemma}\label{Lm:3:1}
 Let $u$ be a locally bounded, local weak sub(super)-solution to \eqref{Eq:1:1} -- \eqref{Eq:1:2p} in $E_T$.
Introduce parameters $\xi$ and $\al$ in $(0,1)$. There exist $\dl$ and $\eps$ in $(0,1)$, depending only on the data and $\al$, such that if
	\begin{equation*}
	\Big|\Big\{
		\pm\big(\boldsymbol \mu^{\pm}-u(\cdot, t_o)\big)\ge \xi\boldsymbol \om
		\Big\}\cap K_{\varrho}(x_o)\Big|
		\ge\al \big|K_{\varrho}\big|,
	\end{equation*}
	then either $$|\boldsymbol \mu^{\pm}|>8\xi\boldsymbol \om$$ or
	\begin{equation*}%\label{Eq:3:1}
	\Big|\Big\{
	\pm\big(\boldsymbol \mu^{\pm}-u(\cdot, t)\big)\ge \eps \xi\boldsymbol \om\Big\} \cap K_{\varrho}(x_o)\Big|
	\ge\frac{\al}2 |K_\varrho|
	\quad\mbox{ for all $t\in\big(t_o,t_o+\dl(\xi\boldsymbol \om)^{q+1-p}\varrho^p\big]$,}
\end{equation*}
provided this cylinder is included in $\mathcal{Q}$.
\end{lemma}
\begin{proof}
The proof of the Lemma \ref{Lm:3:1} can be adapted from the proof of \cite[Lemma~4.1]{BDL} by setting $M = \xi \boldsymbol{\omega}$ and considering the energy estimate for the cylinder $Q = K_\rho(x_o) \times (t_o, t_o + \delta (\xi \boldsymbol{\omega} )^{q+1-p} \rho^p]$ instead of $K_\rho(x_o) \times (t_o, t_o + \delta \rho^p]$.
Further, in \cite{BDL} the term $\mathfrak{g}_\pm$ has been defined for the parameter $q=p-1$.
Thus, in the present situation all computations related to $\mathfrak{g}_\pm$ have to be performed with a general parameter $q > p-1$ instead of $q=p-1$.
\end{proof}

The following measure shrinking lemma holds true for all $p>1$ and $q>p-1$.
%%%%%%%%
\begin{lemma}\label{Lm:3:2}
 Let $u$ be a locally bounded, local weak sub(super)-solution to \eqref{Eq:1:1} -- \eqref{Eq:1:2p} in $E_T$.
Suppose that for some $B>1$, and some $\dl$ and $\xi$ in $(0,1)$, there holds
	\begin{equation*}
	\Big|\Big\{
		\pm\big(\boldsymbol \mu^{\pm}-u(\cdot, t)\big)\ge \xi\boldsymbol \om
		\Big\}\cap K_{B\varrho}(x_o)\Big|
		\ge\al \big|K_{B\varrho}\big|\quad\mbox{ for all $t\in\big(t_o,t_o+\dl\boldsymbol \om^{q+1-p}\varrho^p\big]$,}
	\end{equation*}
Let $\widehat Q_B=K_{B\varrho}(x_o)\times\left(t_o,t_o+\dl\boldsymbol \om^{q+1-p}\rho^p\right]$.
There exist $\varep=B^{\frac{p}{p-q-1}}$ and
 $\boldsymbol \gm>0$ depending only on the data,  $\dl$ and $\alpha$, such that for any positive integer $j_*$,
we have  either $$|\boldsymbol\mu^{\pm}|>\eps \xi\boldsymbol \om 2^{-j_*}$$ or
\begin{equation*}
	\bigg|\bigg\{
	\pm\big(\boldsymbol \mu^{\pm}-u\big)\le\frac{\eps \xi\boldsymbol \om}{2^{j_*}}\bigg\}\cap \widehat Q_B\bigg|
	\le\frac{\boldsymbol\gm}{j_*^{\frac{p-1}p}}|\widehat Q_B|,
\end{equation*}
provided $\widehat{Q}_B$ is included in $\mathcal{Q}$.
\end{lemma}
\begin{proof}
We only show the case of super-solutions, the case of sub-solutions  being similar.
Moreover, we assume $(x_o,t_o)=(0,0)$.
We  employ the energy estimate of Proposition~\ref{Prop:2:1} in $K_{2B\rho}\times (0,\delta\theta\rho^p]$ 
where $\theta=\boldsymbol \om^{q+1-p}$,
with levels
\begin{equation*}
	k_j:=\boldsymbol \mu^-+\frac{\eps \xi\boldsymbol \om}{2^{j}},\quad j=0,1,\cdots, j_*,
\end{equation*}
 and introduce
 a  cutoff function $\z$ in $K_{2B\varrho}$ (independent of $t$) that is equal to $1$ in $K_{B\varrho}$ and vanishes
on $\pl K_{2B\varrho}$, such that $|D\z|\le(B\varrho)^{-1}$.
%and $|\z_t|\le(\theta\rho^2)^{-1}$.
Then, we obtain that
\begin{align*}
	\iint_{\widehat Q_B} |D(u-k_j)_-|^p\,\dx\dt
	&\le
	\int_{K_{2B\varrho}\times\{0\}} \mathfrak g_- (u,k_j) \,\dx\\
	&\quad+ 
	\frac{\boldsymbol\gm}{(B\rho)^p}
	\iint_{K_{2B\rho}\times (0,\delta\theta\rho^p]}(u-k_j)_-^p \,\dx\dt.
\end{align*}
Now we estimate the various terms on the right-hand side.
Let us start with the first one.
When $q\ge 1$, we use $\boldsymbol \mu^-\le u\le k_j$ and $|\boldsymbol\mu^{-}|\le\eps \xi\boldsymbol \om 2^{-j_*}$ to estimate
\begin{equation*}
	\mathfrak g_- (u,k_j)
	=
	q \int_u^{k_j} |s|^{q-1} (s-k_j)_- \,\ds
	\le
	\boldsymbol\gm \max\{|u|,|k_j|\}^{q-1}(u-k_j)_-^2
	\le
	\boldsymbol\gm\left(\frac{\eps \xi\boldsymbol \om}{2^j}\right)^{q+1}
\end{equation*}
for a positive constant $\boldsymbol{\gamma} = \boldsymbol{\gamma}(q)$.
When $0<q\le1$, we use $u\ge \boldsymbol \mu^-$ to obtain
\begin{align*}
	\mathfrak g_- (u,k_j)
	\le
	q (u-k_j)_- \int_{-\frac{1}{2} (u-k_j)_-}^{\frac{1}{2} (u-k_j)_-} |s|^{q-1} \,\ds
	=
	\boldsymbol\gm (u-k_j)_-^{q+1}
	\le
	\boldsymbol\gm\left(\frac{\eps \xi\boldsymbol \om}{2^j}\right)^{q+1},
\end{align*}
for a positive constant $\boldsymbol{\gamma} = \boldsymbol{\gamma}(q)$.
This implies that in any case there holds
\begin{equation*}
\begin{aligned}
	\int_{K_{2B\varrho}\times\{0\}}\mathfrak g_- (u,k_j)\,\dx
	&\le
	 \frac{\boldsymbol \gm}{\dl\theta\varrho^p}
	 \left(\frac{\eps \xi\boldsymbol \om}{2^j}\right)^{q+1} |\widehat Q_B|\\
	 &\le	 \frac{\boldsymbol \gm B^p\varep^{q+1-p}}{\dl(B\varrho)^p}
	 \left(\frac{\eps \xi\boldsymbol \om}{2^j}\right)^{p} |\widehat Q_B|\\
	& \le\frac{\boldsymbol\gm}{\dl(B\varrho)^p}\left(\frac{\eps\xi\boldsymbol \om}{2^j}\right)^p|\widehat Q_B|,
\end{aligned}
\end{equation*}
provided we enforce the condition that $\varep\le B^{\frac{p}{p-q-1}}$.
Here we have used the fact that $q>p-1$.
 In the second integral appearing on the right-hand side of the energy estimate, we may employ the bound
$(u-k_j)_-\le \eps \xi\boldsymbol \om 2^{-j}$ to estimate.
Therefore, in all cases the energy estimate yields that
\begin{equation*}
	\iint_{\widehat Q_B}|D(u-k_j)_-|^p\,\dx\dt\le\frac{\boldsymbol\gm}{\dl(B\varrho)^p}\left(\frac{\eps\xi\boldsymbol \om}{2^j}\right)^p|\widehat Q_B|.
\end{equation*}
Next, we apply \cite[Chapter I, Lemma 2.2]{DB} slicewise 
to $u(\cdot,t)$ for 
$t\in\left(0,\dl\theta\rho^p\right]$
 over the cube $K_\varrho$,
for levels $k_{j+1}<k_{j}$.
Taking into account the measure theoretical information
\begin{equation*}
	\Big|\Big\{u(\cdot, t)>\boldsymbol \mu^-+\eps \xi\boldsymbol \om\Big\}\cap K_{B\varrho}\Big|\ge\al |K_{B\varrho}|
	\qquad\mbox{for all $t\in(0,\dl\theta\rho^p]$,}
\end{equation*}
this gives
\begin{align*}
	&(k_j-k_{j+1})\big|\big\{u(\cdot, t)<k_{j+1} \big\}
	\cap K_{B\varrho}\big|
	\\
	&\qquad\le
	\frac{\boldsymbol\gm (B\varrho)^{N+1}}{\big|\big\{u(\cdot, t)>k_{j}\big\}\cap K_{B\varrho}\big|}	
	\int_{\{ k_{j+1}<u(\cdot,t)<k_{j}\} \cap K_{B\varrho}}|Du(\cdot,t)|\,\dx\\
	&\qquad\le
	\frac{\boldsymbol \gm B\varrho}{\al}
	\bigg[\int_{\{k_{j+1}<u(\cdot,t)<k_{j}\}\cap K_{B\varrho}}|Du(\cdot,t)|^p\,\dx\bigg]^{\frac1p}
	\big|\big\{ k_{j+1}<u(\cdot,t)<k_{j}\big\}\cap K_{B\varrho}\big|^{1-\frac1p}
	\\
	&\qquad=
	\frac{\boldsymbol \gm B\varrho}{\al}
	\bigg[\int_{K_{B\varrho}}|D(u-k_j)_-(\cdot,t)|^p\,\dx\bigg]^{\frac1p}
	\big[ |A_j(t)|-|A_{j+1}(t)|\big]^{1-\frac1p}.
\end{align*}
Here we used in the last line the short hand notation $ A_j(t):= \big\{u(\cdot,t)<k_{j}\big\}\cap K_{B\varrho}$.
We now integrate the last inequality with respect to $t$ over  $(0,\dl\theta\rho^p]$ and apply H\"older's inequality in time.
With the abbreviation $A_j=\{u<k_j\}\cap \widehat Q_B$ this procedure leads to
\begin{align*}
	\frac{\eps \xi\boldsymbol \om}{2^{j+1}}\big|A_{j+1}\big|
	&\le
	\frac{\boldsymbol \gm B\varrho}{\al}\bigg[\iint_{\widehat Q_B}|D(u-k_j)_-|^p\,\dx\dt\bigg]^\frac1p
	\big[|A_j|-|A_{j+1}|\big]^{1-\frac{1}p}\\
	&\le
	\boldsymbol\gm \frac{\eps \xi\boldsymbol \om}{2^j}|\widehat Q_B|^{\frac1p}\big[|A_j|-|A_{j+1}|\big]^{1-\frac{1}p}.
\end{align*}
Here $\boldsymbol\gm$ in the last line has absorbed  
$\delta$ and $\alpha$. %Therefore $\boldsymbol\gm$ depends on only the data and $\alpha$. 
Now take the power $\frac{p}{p-1}$ on both sides of the above inequality to obtain
\[
	\big|A_{j+1}\big|^{\frac{p}{p-1}}
	\le
	\boldsymbol \gm|\widehat Q_B|^{\frac1{p-1}}\big[|A_j|-|A_{j+1}|\big].
\]
Add these inequalities from $0$ to $j_*-1$ to obtain
\[
	j_* |A_{j_*}|^{\frac{p}{p-1}}\le\boldsymbol\gm|\widehat Q_B|^{\frac{p}{p-1}}.
\]
From this we conclude
\[
	|A_{j_*}|\le\frac{\boldsymbol\gm}{j_*^{\frac{p-1}p}}|\widehat Q_B|.
\]
This completes the proof.
\end{proof}

%%%%%%
The next lemma is the main component in the proof of Theorem~\ref{Thm:1:1}.
It translates measure theoretical information into pointwise estimate over {\it expanded} cylinders.
This  is essentially the expansion of positivity for 
the singular, parabolic $p$-Laplacian first established in \cite{ChenDB92}; 
see also \cite{DBGV-pams} and \cite[Proposition~5.1, Chapter~4]{DBGV-mono} for a different proof.
As such it actually holds for $1<p<2$ and $q>0$. 
\begin{lemma}\label{Lm:expansion:p}
Let $u$ be a locally bounded, local, weak sub(super)-solution to \eqref{Eq:1:1} -- \eqref{Eq:1:2p} in $E_T$.
	Introduce the parameters $\Lm,\,c>0$ and $\al\in(0,1)$.
	Suppose that
	\[
	c\boldsymbol\om\le\pm\boldsymbol \mu^{\pm}\le\Lm\boldsymbol\om
	\]
	and for some $0<a \leq \frac12c$,
	\begin{equation*}
		\Big|\Big\{\pm\big(\boldsymbol \mu^{\pm}-u(\cdot, t_o)\big)\ge a\boldsymbol\om\Big\}\cap K_{\varrho}(x_o)\Big|
		\ge
		\al \big|K_\varrho\big|.
		%\quad
		%\mbox{ for all $t_o<t< t_o+\dl\boldsymbol\om^{q+1-p}\varrho^p$}.
	\end{equation*} 
There exist constants $\dl,\,\eta\in(0,1)$ depending on the data, $\Lm$, $c$ and $\al$, such that
\begin{equation*}
	\pm\big(\boldsymbol \mu^{\pm}-u\big)\ge a \eta \boldsymbol\om
	\quad
	\mbox{a.e.~in $K_{2\varrho}(x_o)\times\big\{ t_o+  a^{2-p} \dl \boldsymbol\om^{q+1-p}\varrho^p\big\},$}
\end{equation*}
provided the cylinder $K_{2\varrho}(x_o) \times \big( t_o, t_o+  a^{2-p} \dl \boldsymbol\om^{q+1-p}\varrho^p \big]$ is included in $\mathcal{Q}$.
\end{lemma}
\begin{proof}
We may assume $(x_o,t_o)=(0,0)$ and prove the case of super-solutions only as the other case is similar.
Let $k=\boldsymbol\mu^-+\frac12c\boldsymbol\om$.
By the notion of parabolicity, cf.~\cite[Appendix~A]{BDL} and \cite[Appendix~A]{BDLS-1},  %Section~\ref{S:1:4:2}, 
$u_k:=\min\{u, k\}$ is a local, weak super-solution to \eqref{Eq:1:1}, i.e.
\[
\pl_t \boldsymbol{u_k}^q-\dvg {\bf A}(x,t,u_k,Du_k)\ge0\quad\text{ weakly in }\mathcal{Q}.
\]
To proceed, we define
\[
v:=\boldsymbol{u_k}^q-\boldsymbol{(\mu^-)}^q,
\]
which is non-negative in $\mathcal{Q}$.
Employing the assumption on $\boldsymbol{\mu^-}$, it is straightforward to show that $v$ belongs to the function space \eqref{Eq:1:3p}$_{q=1}$.
Further, $v$ satisfies the equation
\[
v_t-\dvg{\bf \bar{A}}(x,t,v,Dv)\ge0\quad\text{ weakly in }\mathcal{Q}.
\]
Here ${\bf \bar{A}}$ is defined by
\[
	{\bf \bar{A}}(x,t,y,\zeta)
	:=
	{\bf A}\Big(x,t,\big| \widetilde{y} + \boldsymbol{(\mu^-)}^q\big|^{\frac{1-q}q}\big( \widetilde{y} + \boldsymbol{(\mu^-)}^q\big),
	\tfrac1q \big| \widetilde{y}+\boldsymbol{(\mu^-)}^q\big|^\frac{1-q}q \zeta\Big),
\]
where $\widetilde{y}$ denotes the truncation
$$
	\widetilde{y}
	:=
	\min\big\{ \max\{ y, 0 \}, \big(1-\tfrac{1}{2^q}\big) (c\boldsymbol\omega)^q \big\}.
$$
Meanwhile, one verifies that the structure conditions
\begin{equation*}%\label{Eq:1:2p}
	\left\{
	\begin{array}{c}
		{\bf \bar{A}}(x,t,y,\zeta)\cdot \zeta\ge\overline C_o\boldsymbol\omega^{(q-1)(1-p)}|\zeta|^p \\[5pt]
		|{\bf \bar{A}}(x,t,y,\zeta)|\le \overline{C}_1 \boldsymbol\omega^{(q-1)(1-p)} |\zeta|^{p-1},%
	\end{array}
	\right .
	%\qquad \mbox{for a.e.~$(x,t)\in E_T$, $\forall\,u\in\rr$, $\forall\,\z\in\rn$,}
\end{equation*}
with positive constants $\overline C_i =\overline C_i(C_i,p,q,c,\Lambda),\, i=0,1$.

If we further define 
\[
\widetilde{v}(x,t):=v(x,\boldsymbol\om^{(q-1)(p-1)}t),
\]
then it satisfies
\[
\widetilde{v}_t-\dvg{\bf \widetilde{A}}(x,t,\widetilde{v},D\widetilde{v})\ge0\quad\text{ weakly in }
K \times (\boldsymbol \omega^{(q-1)(1-p)} T_1, \boldsymbol \omega^{(q-1)(1-p)} T_2),
\]
where ${\bf \widetilde{A}}(x,t,v,\zeta) := \boldsymbol \omega^{(q-1)(p-1)} {\bf \bar{A}}(x, \boldsymbol \omega^{(q-1)(p-1)} t, v, \zeta)$ is subject to the conditions
\begin{equation*}%\label{Eq:1:2p}
	\left\{
	\begin{array}{c}
		{\bf \widetilde{A}}(x,t,v,\zeta)\cdot Dv\ge \bar{C}_o|\zeta|^p, \\[5pt]
		|{\bf \widetilde{A}}(x,t,v,\zeta)|\le \bar{C}_1|\zeta|^{p-1}.%
	\end{array}
	\right .
	%\qquad \mbox{for a.e.~$(x,t)\in E_T$, $\forall\,v\in\rr$, $\forall\,\z\in\rn$,}
\end{equation*}
In other words, the function $\widetilde{v}$ is a non-negative super-solution
to the parabolic $p$-Laplacian type equation. This allows us to apply the expansion of positivity
in  \cite[Proposition~5.1, Chapter~5]{DBGV-mono}.

By the mean value theorem, the given measure theoretical information in terms of $\widetilde{v}$ becomes
	\begin{equation*}
		\Big|\Big\{\widetilde{v}(\cdot, 0)\ge \widetilde{a}\boldsymbol\om^q\Big\}\cap K_{\varrho}\Big|
		\ge
		\al \big|K_\varrho\big|,
		%\quad
		%\mbox{ for all $0<t<\dl\boldsymbol\om^{q(2-p)}\varrho^p$},
		\label{eq:transformed info}
	\end{equation*}
where $\widetilde{a}=a\widetilde{\boldsymbol\gm}$ for some positive $\widetilde{\boldsymbol\gm}=\widetilde{\boldsymbol\gm}(q,c,\Lm)$.
An application of \cite[Proposition~5.1, Chapter~5]{DBGV-mono} to $\widetilde{v}$ yields that for some positive $\eta$ and $\dl$
 depending only on the data, $\al$, $c$ and $\Lm$, such that
\[
\widetilde{v}\ge\eta \widetilde{a} \boldsymbol\om^q\quad\text{ a.e. in }K_{2\rho}\times\big\{ \dl (\widetilde{a}\boldsymbol\om^q)^{2-p} \rho^p\big\}.
%K_{2\rho}\times\big(\kappa \dl (\widetilde{a}\boldsymbol\om^q)^{2-p} \rho^p, \dl(\widetilde{a}\boldsymbol\om^q)^{2-p} \rho^p\big].
\]
Reverting to the original function $u$ with the aid of the mean value theorem, we obtain that
$$
	u - \boldsymbol\mu^-
	\geq
	\widetilde{\boldsymbol\gm} \eta a \boldsymbol \omega
	\quad \text{a.e.~in }
	K_{2 \rho} \times\big\{ \dl (a\widetilde{\boldsymbol\gm})^{2-p} \boldsymbol \omega^{q+1-p} \rho^p \big\}.
$$
Redefining $\widetilde{\boldsymbol\gm} \eta$ as $\eta$ and $\dl \widetilde{\boldsymbol\gm}^{2-p}$ as $\dl$, we may conclude.
\end{proof}

\begin{remark}
\label{rem:kappa}\upshape
Observe that we may replace $a$ in Lemma~\ref{Lm:expansion:p} by $\kappa a$ for any $\kappa\in(0,1)$.
As a result, under the same hypotheses, we obtain $\pm \big( \boldsymbol \mu^\pm - u(\cdot, t) \big) \ge \kappa a \eta \boldsymbol \om$ a.e. in $K_{2\rho}$ for all times
\[
t\in \big[ t_o+  \kappa^{2-p} a^{2-p} \dl \boldsymbol\om^{q+1-p}\varrho^p , t_o+  a^{2-p} \dl \boldsymbol\om^{q+1-p}\varrho^p\big].
\]
This means the pointwise positivity can be claimed as close to $t_o$ as we want.
\end{remark}
%%%%%%%%%%%%%%
%%%%%%%%%%%%%%%%%%%%%%%%%%%%%%
%%%%%%%%%%%%%%%%%%%%%%%%%%%%%%
\section{Proof of Theorem \ref{Thm:1:1}}\label{S:5}
\subsection{The proof begins} 
Assume without loss of generality that $(x_o,t_o)=(0,0)$,
introduce $Q_o=Q(\rho^{\frac12}, \rho^{p})$  and
set
\begin{equation*}
	\boldsymbol \mu^+=\essup_{Q_o}u,
	\quad
	\boldsymbol\mu^-=\essinf_{Q_o}u,
	\quad
	\boldsymbol\om \geq \boldsymbol\mu^+-\boldsymbol\mu^-.
\end{equation*}
Let $\theta=(\frac14\boldsymbol\om)^{q+1-p}$ and consider a radius $\rho$ such that $0 < \rho\le A^{-2}$ for some $A>1$ to be determined in terms of the data.
At this stage, we may assume that $\boldsymbol\om\le 1$. Then the following intrinsic relation holds true:
\begin{equation}\label{Eq:start-cylinder}
Q(A\rho, \theta\rho^p)\subset Q_o,\quad\text{ such that }\quad \essosc_{Q(A\rho, \theta\rho^p)}u\le\boldsymbol\om.
\end{equation}

Our proof unfolds along two main
cases, namely for some $\xi\in(0,1)$ to be determined,
\begin{equation}\label{Eq:Hp-main}
\left\{
\begin{array}{c}
\mbox{when $u$ is near zero:  $\boldsymbol\mu^-\le\xi\boldsymbol\om$ and 
$\boldsymbol\mu^+\ge-\xi\boldsymbol\om$};\\[5pt]
\mbox{when $u$ is away from zero: $\boldsymbol\mu^->\xi\boldsymbol\om$ or $\boldsymbol\mu^+<-\xi\boldsymbol\om$.}
\end{array}\right.
\end{equation}
Note that \eqref{Eq:Hp-main}$_1$ implies that  
$|\boldsymbol\mu^\pm|\le2\boldsymbol\om$.
When this case  holds, we deal with it next in Sections~\ref{S:case-1} -- \ref{S:case-1-3};
the other case will be treated in Section~\ref{S:case-2}.
%%%%%
\subsection{Reduction of oscillation near zero--Part I}\label{S:case-1}
In this section, we will assume that \eqref{Eq:Hp-main}$_1$ holds true.
To proceed, we may further observe that
one of the following must hold true:
Either we have that $  \boldsymbol\mu^+-\boldsymbol\mu^- \leq \tfrac12 \boldsymbol \omega$, which will be considered later on, or $\boldsymbol\mu^+-\boldsymbol\mu^- > \tfrac12 \boldsymbol \omega $, which we assume in Sections~\ref{S:case-1} -- \ref{S:case-1-2}.
Consequently, there holds
\begin{equation}\label{Eq:mu-pm}
\boldsymbol\mu^+-\tfrac18\boldsymbol\om\ge\tfrac18\boldsymbol\om\quad\text{ or }\quad \boldsymbol\mu^-+\tfrac18\boldsymbol\om\le-\tfrac18\boldsymbol\om.
\end{equation}
Let us take for instance the second one, i.e.~\eqref{Eq:mu-pm}$_2$, as the other one can be treated analogously.
Hence we may assume that $ -2\boldsymbol\om\le\boldsymbol\mu^-\le -\frac14\boldsymbol\om$ and Lemma~\ref{Lm:expansion:p} is at our disposal.

Suppose $A$ is a large integer with no loss of generality, and divide $Q(A\rho, \theta\rho^p)$ into $A^N$ disjoint but adjacent sub-cylinders
$$Q_i:=K_{\rho}(x_i)\times(- \theta\rho^p,0]\quad\text{ for }i=1,\cdots, A^N.$$
One of the following two alternatives must hold true:
\begin{equation}\label{Eq:alternative-int}
\left\{
\begin{array}{ll}
\Big|\Big\{u\le\boldsymbol\mu^-+\tfrac14\boldsymbol\om\Big\}\cap Q_i\Big|\le \nu|Q_i|\quad\text{ for some }i\in\big\{1,\cdots, A^N\big\},\\[5pt]
\Big|\Big\{u\le\boldsymbol\mu^-+\tfrac14\boldsymbol\om\Big\}\cap Q_i\Big|> \nu |Q_i|\quad\text{ for any }i\in\big\{1,\cdots, A^N\big\}.
\end{array}\right.
\end{equation}
Here the number $\nu\in(0,1)$ is determined in Lemma~\ref{Lm:DG:1} in terms of the data.

First suppose \eqref{Eq:alternative-int}$_1$ holds true for some $i$. An application of Lemma~\ref{Lm:DG:1} (with $\xi=\frac14$)
gives us that, recalling $|\boldsymbol\mu^-|\le2\boldsymbol\om$ due to \eqref{Eq:Hp-main}$_1$,
\begin{equation}\label{Eq:initial-condi}
u\ge \boldsymbol\mu^-+\tfrac18\boldsymbol\om\quad\text{ a.e. in }\tfrac12Q_i.
\end{equation}
Here the notation $\tfrac12Q_i$ should be self-explanatory in view of Lemma~\ref{Lm:DG:1}.
Recalling that we have assumed $ -2\boldsymbol\om\le\boldsymbol\mu^-\le -\frac14\boldsymbol\om$,
after $n+1$ applications of Lemma~\ref{Lm:expansion:p}, $n \in \nn$,
with $a = \frac{1}{8} \kappa \eta^{i-1}$ in the $i$-th step, $i=1, \ldots, n+1$,
the pointwise information \eqref{Eq:initial-condi} yields
$$
	u \ge \boldsymbol\mu^- + \tfrac18 \kappa \eta^{n+1} \boldsymbol\om
$$
a.e.~in
$$
	K_{2^n \rho}(x_i) \times
	\Big(-\theta \big( \tfrac{1}{2}\rho \big)^p
	+ \big( \tfrac18 \kappa\big)^{2-p} \delta \boldsymbol \omega^{q+1-p} \sum_{j=0}^n \eta^{(2-p)j} 2^{(j-1) p} \rho^p,
	0 \Big],
$$
where $\kappa \in (0,1)$ is a parameter to be chosen and 
$\delta,\,\eta \in (0,1)$ depend only on the data.
Now, we take $n$ so large that $2^n\ge2A$.
Then we choose $\kappa$ so small that
$$
	-\theta \left( \tfrac{1}{2}\rho \right)^p
	+ \big( \tfrac18 \kappa\big)^{2-p} \delta \boldsymbol \omega^{q+1-p} \sum_{j=0}^n 2^{(j-1) p} \rho^p
	\leq
	-  \theta \big(\tfrac14\rho\big)^p,
	\quad \text{ i.e. }\quad
	\kappa^{2-p}
	\leq
	\boldsymbol\gm(n, p, q) \dl^{-1}.
$$
This fixes $\eta_1 := \tfrac18 \kappa \eta^{n+1}$ depending only on the data and $A$. Moreover, we obtain that
\[
u\ge\boldsymbol\mu^-+\eta_1\boldsymbol\om\quad\text{ a.e. in }\widetilde{Q}=K_{A\rho}\times \big(- \theta\big(\tfrac14\rho\big)^p,0\big],
\]
which in turn yields a reduction of oscillation
\begin{equation}\label{Eq:reduc-osc-int-1}
\essosc_{\widetilde{Q}}u\le (1-\eta_1)\boldsymbol\om.
\end{equation}
Keep in mind that $A>1$ is yet to be determined in terms of the data.

%%%%%%%%%%%%%%%%%
\subsection{Reduction of oscillation near zero--Part II}\label{S:case-1-2}
In this section, we still assume that \eqref{Eq:Hp-main}$_1$ holds true.
However, we turn our attention to the second alternative \eqref{Eq:alternative-int}$_2$.
Observing that $\boldsymbol\mu^+-\frac14\boldsymbol\om\ge\boldsymbol\mu^-+\frac14\boldsymbol\om$ due to our assumption for Sections~\ref{S:case-1} -- \ref{S:case-1-2},
%$\boldsymbol\mu^+-\boldsymbol\mu^- \geq \frac{1}{2} \boldsymbol \omega$,
  we may rephrase \eqref{Eq:alternative-int}$_2$ as
\[
\Big|\Big\{\boldsymbol\mu^+-u\ge\tfrac14\boldsymbol\om\Big\}\cap Q_i\Big|> \nu |Q_i|\quad\text{ for any }i\in\big\{1,\cdots, A^N\big\}.
\]
Letting $B\le A$ be another positive integer to be determined, we combine $B^N$ adjacent blocks of $\{Q_i\}$ to obtain that
\begin{equation}\label{Eq:measure-B}
\Big|\Big\{\boldsymbol\mu^+-u\ge\tfrac14\boldsymbol\om\Big\}\cap Q(B\rho, \theta\rho^p)\Big|> \nu |Q(B\rho, \theta\rho^p)|.
%\quad\text{ for any }i\in\big\{1,\cdots, A^N\big\}.
\end{equation}
Here we have  omitted the vertex of $Q(B\rho, \theta\rho^p)$ for simplicity
while bearing in mind that it is arbitrary as long as such a cylinder is included in $Q(A\rho, \theta\rho^p)$.

Then from \eqref{Eq:measure-B} we find some 
$$t_*\in \big[-\theta\rho^p, - \tfrac12 \nu\theta\rho^p\big]$$ 
such that
\begin{equation*}%\label{Eq:measure-B-slice}
\Big|\Big\{\boldsymbol\mu^+-u(\cdot, t_*)\ge\tfrac14\boldsymbol\om\Big\}\cap K_{B\rho}\Big|>\tfrac12\nu |K_{B\rho}|,
\end{equation*}
since otherwise we could compute
\begin{align*}
	\Big|\Big\{\boldsymbol\mu^+-u\ge\tfrac14\boldsymbol\om\Big\}
	&\cap Q(B\rho, \theta\rho^p)\Big| \\
	&=
	\int_{-\theta \rho^p}^{-\frac{1}{2} \nu \theta \rho^p}
	\Big|\Big\{\boldsymbol\mu^+-u(\cdot, t) \ge\tfrac14\boldsymbol\om\Big\}
	\cap K_{B \rho} \Big| \,\dx\dt \\
	&\hphantom{=}
	+\int_{-\frac{1}{2} \nu \theta \rho^p}^0
	\Big|\Big\{\boldsymbol\mu^+-u(\cdot, t) \ge\tfrac14\boldsymbol\om\Big\}
	\cap K_{B \rho} \Big| \,\dx\dt \\
	&\leq
	\tfrac{1}{2} \nu \Big(\theta \rho^p - \tfrac12 \nu \theta \rho^p \Big) |K_{B \rho}|
	+\tfrac12 \nu \theta \rho^p |K_{B\rho}| \\
	&<
	\nu |Q(B\rho, \theta\rho^p)|
\end{align*}
in contradiction to \eqref{Eq:measure-B}.
Note again that the center of the cube $K_{B\rho}$ is omitted.
This serves as the initial measure information for an application of Lemma~\ref{Lm:3:1} with $\xi=\frac14$.
Indeed, keeping in mind that $|\boldsymbol\mu^+|\le2\boldsymbol\om$,
there exist $\dl_1$ and $\varep$ in $ (0,1)$ depending on the data and $\nu$, such that
\begin{equation}\label{Eq:measure}
\Big|\Big\{\boldsymbol\mu^+-u(\cdot, t)\ge\tfrac14\varep\boldsymbol\om\Big\}\cap K_{B\rho}\Big|>\tfrac14\nu |K_{B\rho}|
\end{equation}
for all times
\[
t_*\le t\le t_*+\dl_1\theta (B\rho)^p.
\]
If we choose the integer $ B \geq \dl_1^{-\frac1{p}}$, then \eqref{Eq:measure} holds for all time
\[
- \tfrac12 \nu\theta\rho^p\le t\le 0.
\]

Next, the measure information in \eqref{Eq:measure} allows us to apply Lemma~\ref{Lm:3:2}.
Indeed, for some $\boldsymbol\gm>1$ and $\widetilde\varep\in(0,1)$ depending on the data (remembering that
all parameters $\{\nu,\dl_1,\varep,B\}$ have been determined by the data), and for all $j_*\in\nn$, there holds
either $|\boldsymbol\mu^{+}|>\widetilde\varep \varep\boldsymbol \om 2^{-j_*-2}$ or
\begin{equation*}
	\bigg|\bigg\{
	\boldsymbol \mu^{+}-u\le\frac{\widetilde\varep\varep\boldsymbol \om}{2^{j_*+2}}\bigg\}\cap \widehat Q_B\bigg|
	\le\frac{\boldsymbol\gm}{j_*^{\frac{p-1}p}}|\widehat Q_B|,
\end{equation*}
where we have set
\[
\widehat{Q}_B=K_{B\rho}\times\big(-\tfrac12\nu\theta\rho^p,0\big].
\]
We may assume that $A_1:=L_1B\le A$ for some large integer $L_1$ to be determined.
By the arbitrariness of the vertex of $\widehat{Q}_B$, we may combine $L_1^N$
disjoint but adjacent blocks of $\widehat{Q}_B$ that make up $\widehat{Q}_{A_1}$,
 where we have set
\[
\widehat{Q}_{A_1}=K_{A_1\rho}\times\big(-\tfrac12\nu\theta\rho^p,0\big].
\]
Note carefully that now the vertex of $\widehat{Q}_{A_1}$ is fixed at the origin.
 Then the above measure estimate actually yields that
 either $|\boldsymbol\mu^{+}|>\widetilde\varep \varep\boldsymbol \om 2^{-j_*-2}$ or
\begin{equation*}
	\bigg|\bigg\{
	\boldsymbol \mu^{+}-u\le\frac{\widetilde\varep\eps\boldsymbol \om}{2^{j_*+2}}\bigg\}\cap \widehat Q_{A_1}\bigg|
	\le\frac{\boldsymbol\gm}{j_*^{\frac{p-1}p}}|\widehat Q_{A_1}|.
\end{equation*}

From this measure estimate, the last step in order to reduce the oscillation in this case lies in applying Lemma~\ref{Lm:DG:1}.
In fact, fixing $\nu$ as in Lemma~\ref{Lm:DG:1}, we first choose $j_*$ to satisfy that $$\boldsymbol\gm j_*^{-\frac{p-1}p}\le\nu.$$
Then we choose $A_1$ (and hence $L_1$) to fulfill the intrinsic relation required by Lemma~\ref{Lm:DG:1}:
\begin{equation}\label{Eq:choice-A-1}
\Big(\frac{\widetilde\varep\eps\boldsymbol \om}{2^{j_*+2}}\Big)^{q+1-p}(A_1\rho)^p=\tfrac12\nu\theta\rho^p,\quad\text{ i.e. }\quad
A_1^p=(L_1B)^p=\tfrac12\nu\Big(\frac{\widetilde\varep\eps}{2^{j_*+2}}\Big)^{p-q-1}.
\end{equation}
As a result, if we set
\begin{equation}
	\label{Eq:choice-xi}
	\xi=\widetilde\varep\varep 2^{-j_*-2},
\end{equation}
and impose $|\boldsymbol \mu^{+}|\le\xi\boldsymbol \om$,
then Lemma~\ref{Lm:DG:1} implies that
\[
\boldsymbol \mu^{+}-u\ge\tfrac12\xi\boldsymbol \om\quad\text{ a.e. in }\tfrac12\widehat{Q}_{A_1},
\]
which in turn gives that
\begin{equation}\label{Eq:reduc-osc-int-2}
\essosc_{\frac12\widehat{Q}_{A_1}}u\le (1-\tfrac12\xi)\boldsymbol \om.
\end{equation}
Here the notation $\tfrac12\widehat{Q}_{A_1}$ should be self-explanatory in view of Lemma~\ref{Lm:DG:1}.

The case $\boldsymbol \mu^{+}\le-\xi\boldsymbol \om$ will be treated in Section~\ref{S:case-2},
whereas the case $\xi\boldsymbol \om\le\boldsymbol \mu^{+}\le 2\boldsymbol \om$ is handled with the aid of 
Lemma~\ref{Lm:expansion:p}. In fact, let $A_o:=L_oB\le A$ for some large integer $L_o$
and recall the measure information \eqref{Eq:measure-B} in $Q_\rho(\theta\rho^p, B\rho)$. 
By the arbitrariness of the vertex of $Q(B\rho, \theta\rho^p)$, we may combine $L_o^N$
disjoint but adjacent blocks of $Q(B\rho, \theta\rho^p)$ that make up $Q(A_o\rho, \theta\rho^p)$.
Keep in mind that now the vertex of $Q(A_o\rho, \theta\rho^p)$ is fixed at the origin.
Then the measure information \eqref{Eq:measure-B} may be rephrased in $Q(A_o\rho, \theta\rho^p)$
and  consequently 
 it is not hard to see that there exists 
$$t_*\in \big[-\theta\rho^p, - \tfrac34 \nu\theta\rho^p\big],$$ 
such that
\begin{equation*}%\label{Eq:measure-B-slice}
\Big|\Big\{\boldsymbol\mu^+-u(\cdot, t_*)\ge\tfrac14\boldsymbol\om\Big\}\cap K_{A_o\rho}\Big|>\tfrac14 \nu |K_{A_o\rho}|.
\end{equation*}
This measure information allows us to apply Lemma~\ref{Lm:expansion:p} with the choice $a = \frac{\xi}{2} \leq \xi = c$ and Remark \ref{rem:kappa}.
Indeed, there exist positive $\dl_o$ and $\eta_o$ depending on the data, $\xi$ and $\nu$,
such that for any $\kappa\in(0,1)$ we have
\[
\boldsymbol\mu^+-u\ge \kappa \eta_o\boldsymbol\om\quad\text{ a.e. in }
K_{A_o\rho}\times\big(t_*+\kappa^{2-p} \dl_o\boldsymbol\om^{q+1-p}(A_o\rho)^p, t_*+\dl_o\boldsymbol\om^{q+1-p}(A_o\rho)^p\big].
\]
We select the integer $A_o$ (and hence $L_o$) such that
\begin{equation}\label{Eq:choice-A-2}
	(\tfrac14 \boldsymbol\om)^{q+1-p}\rho^p\le\dl_o\boldsymbol\om^{q+1-p}(A_o\rho)^p,
	\quad\text{ i.e. }\quad
	A_o^p=(L_oB)^p \geq \dl_o^{-1}4^{p-q-1}.
\end{equation}
Bearing in mind that the choice of $A_1$ in \eqref{Eq:choice-A-1} and  $A_o$ in \eqref{Eq:choice-A-2}
can be made even larger by properly adjusting relevant parameters that determine them,
the final choice of $A$ is the larger one of $A_o$ and $A_1$.
This remark tacitly used the fact that $0<p-1<q$ and $1<p<2$.
Next, we choose $\kappa$ such that
$$
	- \tfrac34 \nu\theta\rho^p + \kappa^{2-p} \dl_o\boldsymbol\om^{q+1-p}(A_o\rho)^p
	\leq
	-\tfrac12 \nu \theta \rho^p,
	\quad \text{ i.e. } \quad
	\kappa^{2-p} = \frac{\nu}{4^{q+2-p} \delta_o A_o^p}.
$$
Note that $\kappa$ depends only on the data, since $\nu$, $\delta_o$ and $A_o$ have already been determined by the data. 
As a result, we arrive at 
\begin{equation}\label{Eq:reduc-osc-int-3}
\essosc_{\widehat{Q}_{A_o}}u\le (1-\kappa\eta_o)\boldsymbol\om,
\end{equation}
 where we have set
\[
\widehat{Q}_{A_o}=K_{A_o\rho}\times\big(-\tfrac12\nu\theta\rho^p,0\big].
\]

To summarize, let us define 
\[
\eta=\min\big\{\eta_1,\,\kappa\eta_o,\,\tfrac12\xi\big\}, %\,\tfrac12\xi\Big\},\quad \bar\nu=\Big\{A^{-1}, \tfrac12\nu\Big\},
\]
where
 $\eta_1$ is as in \eqref{Eq:reduc-osc-int-1},
$\kappa\eta_o$ is as in \eqref{Eq:reduc-osc-int-3} and $\frac12\xi$ is as in \eqref{Eq:reduc-osc-int-2}. 
Combining \eqref{Eq:reduc-osc-int-1}, \eqref{Eq:reduc-osc-int-3} and \eqref{Eq:reduc-osc-int-2} and taking into account the violation of \eqref{Eq:mu-pm}, that is, $  \boldsymbol\mu^+-\boldsymbol\mu^- \leq \tfrac12 \boldsymbol \omega$, gives
the reduction of oscillation
 \begin{equation}\label{Eq:reduc-osc-int-4}
\essosc_{Q_{\frac14\rho}(\nu\theta)}u\le (1-\eta)\boldsymbol\om,
 \end{equation}
 provided the intrinsic relation \eqref{Eq:start-cylinder} is verified.

In order to iterate the above argument, we introduce
\[
\boldsymbol\om_1=(1-\eta)\boldsymbol\om; %\max\Big\{(1-\eta)\boldsymbol\om, L\rho^{\frac1{p-q-1}}\Big\};
\]
we need to choose $\rho_1=\lm\rho$ for some $\lm\in(0,1)$, such that
\[
Q(\theta_1\rho_1^p, A\rho_1)\subset Q_{\frac14\rho}(\nu\theta),%\quad Q_{\rho_1}(A\theta_1)\subset Q_o,
\quad\text{ where }\theta_1=(\tfrac14\boldsymbol\om_1)^{q+1-p}.
\]
This amounts to requiring that
\[
\theta_1\rho_1^p\le \big(\tfrac14 \boldsymbol\om\big)^{q+1-p}(\lm\rho)^p\le \nu(\tfrac14\boldsymbol\om)^{q+1-p}(\tfrac14\rho)^p,
\quad A\rho_1=A\lm\rho\le\tfrac14\rho;
\]
consequently, the  set inclusion holds if we take
\[
\lm= \min\Big\{\tfrac1{4}A^{-1},\,\tfrac14  \nu^{\frac{1}p}\Big\}.
\]
Therefore,  by this set inclusion and \eqref{Eq:reduc-osc-int-4}, 
we have arrived at the intrinsic relation
\[
\essosc_{Q(A\rho_1, \theta_1\rho_1^p)}u\le \boldsymbol\om_1,
\]
which takes the place of \eqref{Eq:start-cylinder} in the next stage.
%%%

\subsection{Reduction of oscillation near zero concluded}\label{S:case-1-3}
Now we may proceed by induction. 
Suppose that, up to $i=1,2,\cdots j-1$, we have built
\begin{equation*}
\left\{
	\begin{array}{c}
	\rho_o=\rho,
	\quad
	\dsty\varrho_i=\lm\varrho_{i-1},
	\quad
	\theta_i=\big(\tfrac14\boldsymbol\om_i\big)^{q+1-p}\\[5pt]
	\boldsymbol\om_o=\boldsymbol\om,
	\quad
	\boldsymbol\om_i= (1-\eta)\boldsymbol\om_{i-1},\\[5pt]
	Q_i=Q(A\rho_i, \theta_i\rho_i^p),
	\quad
	Q'_i=Q_{\frac14\rho_i}(\nu\theta_i)\\[5pt]
	\dsty\boldsymbol\mu_i^+=\essup_{Q_i}u,
	\quad
	\boldsymbol\mu_i^-=\essinf_{Q_i}u,
	\quad
	\essosc_{Q_i}u\le\boldsymbol\om_i.
	\end{array}
\right.
\end{equation*}
For all the indices $i=1,2,\cdots j-1$, we alway assume that  \eqref{Eq:Hp-main}$_1$ holds, i.e.,
$$
	\boldsymbol\mu_i^-\le\xi\boldsymbol\om_i\quad
	\text{ and }\quad\boldsymbol\mu_i^+\ge-\xi\boldsymbol\om_i.
$$
In this way, one can repeat the argument at the beginning  and conclude for all $i=1,2,\cdots j$,
\begin{equation}
	Q_i\subset Q'_{i-1},
	\quad
	\essosc_{Q_i} u \le (1-\eta)\boldsymbol\om_{i-1} = \boldsymbol\om_i.
	\label{Eq:case1}
\end{equation}
%%%%%%%%
\subsection{Reduction of Oscillation Away From Zero}\label{S:case-2}
In this section we assume that there exists an index $j$ such that the second case in \eqref{Eq:Hp-main} is satisfied, i.e.
\[
\text{either \quad$\boldsymbol\mu_j^-> \xi\boldsymbol \om_j$ \quad or\quad $\boldsymbol\mu_j^+< -\xi\boldsymbol \om_j$.}
\]
Since the other case is analogous, we only treat $\boldsymbol\mu_j^->\xi\boldsymbol\om_j$.
If we assume that $j$ is the first index fulfilling \eqref{Eq:Hp-main}$_2$, we have that $\boldsymbol\mu_{j-1}^-\le \xi\boldsymbol \om_{j-1}$.
Therefore, since $Q_j \subset Q_{j-1}$ and by definition of the essential supremum we obtain
\[
	\boldsymbol\mu_{j}^-
	\leq
	\boldsymbol \mu_j^+
	\leq
	\boldsymbol\mu_{j-1}^+
	\leq
	\boldsymbol\mu_{j-1}^- + \boldsymbol \om_{j-1}
	\leq
	(1 + \xi) \boldsymbol \omega_{j-1}
	\leq
	\frac{1 + \xi}{1 - \eta} \boldsymbol \om_{j}.
\]
Consequently, we may estimate
\begin{equation}\label{Eq:6:9}
\xi\boldsymbol \om_{j}\le\boldsymbol \mu_{j}^-\le\frac{1+\xi}{1-\eta}\boldsymbol \om_{j}.
\end{equation}
Hence, $u$ is bounded away from zero in $Q_j$ and thus starting from $j$, the equation \eqref{Eq:1:1} resembles the parabolic $p$-Laplacian type equation in $Q_j$.
In the following, we drop the index $j$ from our notation for simplicity,
and introduce $v:=u/\boldsymbol\mu^-$ in $Q\equiv Q(A \rho, \theta \rho^p)$.
It is straightforward to verify that $v$ satisfies
\begin{equation}\label{Eq:6:10}
1\le v\le\frac{1+\xi}{\xi}\quad\text{ a.e. in }Q
\end{equation}
and solves the equation
\begin{equation*}
	\pl_tv^{q}-\dvg\bar{\bl{A}}(x,t,v, Dv)=0\quad\text{ weakly in }Q,
\end{equation*}
where, for $ (x,t)\in Q$, $y \in\rr$ and $\z\in\rn$, we have defined 
$$
	\bar{\bl{A}}(x,t,y, \z)
	=
	 \bl{A}(x,t,\boldsymbol\mu^- y, \boldsymbol\mu^- \z) /(\boldsymbol\mu^-)^{q},
$$
which is subject to the structure conditions
\begin{equation*}
	\left\{
	\begin{array}{c}
		\bar{\bl{A}}(x,t,y,\z) \cdot \z \ge C_o (\boldsymbol \mu^-)^{p-q-1} |\z|^p \\[5pt]
		|\bar{\bl{A}}(x,t,y,\z)| \le C_1 (\boldsymbol \mu^-)^{p-q-1} |\z|^{p-1}
	\end{array}
	\right .
	\qquad \mbox{ for a.e.~$(x,t)\in Q$, $\forall\, y\in\rr$, $\forall\, \z\in\rn$.}
\end{equation*}
Next, we consider the equation satisfied by $w:=v^{q}$, i.e.
\begin{equation*}
\partial_tw-\dvg\widetilde{\bl{A}}(x,t,w, Dw)=0\quad\text{ weakly in }Q,
\end{equation*}
where the vector-field $\widetilde{\bf A}$ is defined by
\[
	\widetilde{\bl{A}}(x,t,y,\zeta)
	=\
	\bar{\bl{A}}\Big( x,t,\widetilde y^{\frac{1}{q}},\tfrac{1}{q} \widetilde y^{\frac{1-q}{q}}
	\zeta\Big),
\]
for a.e.~$(x,t)\in Q$, any $y \in\rr$ and any $\zeta\in\rn$.
Here, $\widetilde y$ is defined by
\[
	\widetilde y :=
	\min\Big\{ \max\big\{y,\tfrac12\big\}, 2\Big(\frac{1+\xi}{\xi}\Big)^q \Big\}.
\]
Due to \eqref{Eq:6:10}, $w$ is contained in the function space \eqref{Eq:1:3p}$_{q=1}$.
Furthermore, taking into account \eqref{Eq:6:10}, a straightforward calculation shows that there exist positive constants $\widetilde{C}_o=\boldsymbol\gamma_o (p,q,\xi)C_o$ and $\widetilde{C}_1= \boldsymbol\gamma_1 (p,q,\xi)C_1$, such that
\begin{equation*}
		\widetilde{\bl{A}}(x,t,y,\zeta)\cdot \zeta
		\geq
		\widetilde{C}_o (\boldsymbol \mu^-)^{p-q-1} |\zeta|^p 
		\quad\mbox{and}
		\quad
		|\widetilde{\bl{A}}(x,t,y,\zeta)|
		\leq
		\widetilde{C}_1 (\boldsymbol \mu^-)^{p-q-1} |\zeta|^{p-1},
\end{equation*}
for a.e.~$(x,t)\in Q$, any $y\in\rr$, and any $\zeta\in\rn$.
Keep in mind that $\xi$ is already fixed in \eqref{Eq:choice-xi} in dependence of the data.
In order to eliminate $\boldsymbol \mu^-$ in the preceding structure conditions, define
$$
	\widehat{w}(x,t) := w(x, (\boldsymbol \mu^-)^{q+1-p} t),
$$
which solves the equation
\begin{equation}
	\partial_t \widehat{w} - \dvg \widehat{\bl{A}} (x, t, \widehat{w}, D\widehat{w} ) = 0
	\quad \text{weakly in }
	\widehat{Q} :=
	K_{A\rho} \times (- (\boldsymbol \mu^-)^{p-q-1} \theta \rho^p, 0].
	\label{eq:pLaplace equation}
\end{equation}
Here, the vector field $\widehat{\bl{A}}$ given by
\begin{equation*}
	\widehat{\bl{A}}(x, t, v, \zeta)
	:=
	(\boldsymbol \mu^-)^{q+1-p} \widetilde{\bl{A}} (x, (\boldsymbol \mu^-)^{q+1-p} t, v, \zeta)
\end{equation*}
satisfies the desired structure conditions
\begin{equation}
	\left\{
	\begin{array}{c}
		\widehat{\bl{A}}(x,t,y,\z) \cdot \z \ge \widetilde{C}_o |\z|^p \\[5pt]
		|\widehat{\bl{A}}(x,t,y,\z)| \le \widetilde{C}_1 |\z|^{p-1}
	\end{array}
	\right.
	\qquad \mbox{ for a.e.~$(x,t)\in \widehat{Q}$, $\forall\, y\in\rr$, $\forall\, \z\in\rn$.}
	\label{eq:pLaplace structure}
\end{equation}
Thus, $\widehat{w}$ is a weak solution of an equation of parabolic $p$-Laplace type with $1<p<2$.
Consequently, the following well-known oscillation decay estimate first proved in \cite{ChenDB92} is available.
While the original proof uses a different scaling, we adapt the statement as indicated in \cite[Chapter~IV, Remark~2.2]{DB}, such that it fits our geometry.
\begin{proposition}\label{Prop:5:1}
Let $1 < p < 2$.
Suppose $\widehat{w}$ is a bounded, local, weak solution to 
\eqref{eq:pLaplace equation} -- \eqref{eq:pLaplace structure} in $\widehat{Q}$.
If for some constants $\sig$ in $(0,1)$ and $\boldsymbol{\widehat{\om}}>0$, there holds
\begin{equation}\label{Eq:5:8}
	\essosc_{Q_{\sig  \varrho}(\widehat\theta)} \widehat{w}
	\le
	\boldsymbol{\widehat{\om}}
	\quad
	\text{where }
	\widehat\theta = \boldsymbol{\widehat{\om}}^{2-p}
	\text{ and }
	Q_{\sig  \varrho}(\widehat\theta )\subset \widehat{Q},
\end{equation}
then, there exist constants $\beta_1$ in $(0,1)$ and $\boldsymbol\gm>1$ depending only on the data
$\{N,p,\widetilde C_o, \widetilde C_1\}$, $A$ and $\sig$, such that for all $0<r< \rho$, we have
\[
	\essosc_{Q_r(\widehat\theta)} \widehat{w}
	\le
	\boldsymbol\gm \boldsymbol{\widehat{\om}} \left(\frac{r}{\rho}\right)^{\be_1}.
\]
\end{proposition}
In order to find $\sigma \in (0,1)$ and $\boldsymbol{\widehat{\om}} > 0$ such that \eqref{Eq:5:8} is fulfilled, by the mean value theorem, \eqref{Eq:6:9} and \eqref{Eq:6:10} we first compute
\begin{align*}
	\essosc_{\widehat{Q}} \widehat{w}
	&=
	\essosc_Q w
	\leq
	%\boldsymbol\gamma(q) 
	q\max \big\{|\essup_Q v|^{q-1}, |\essinf_Q v|^{q-1} \big\} \essosc_Q v \\
	&\leq
	q \max \Big\{ 1, \Big( \frac{1+\xi}{\xi} \Big)^{q-1} \Big\} \frac{\boldsymbol \omega}{\boldsymbol \mu^-}
	\leq
	q \max \Big\{ 1, \Big( \frac{1+\xi}{\xi} \Big)^{q-1} \Big\} \frac{1}{\xi}
	=:
	\widehat{\boldsymbol\omega},
\end{align*}
where $\widehat{\boldsymbol \omega}$ depends only on the data.
Next, we have that
$Q_{\sigma \rho}(\widehat\theta) \subset \widehat{Q}$
if
$$
	\sigma^p
	\leq
	\widehat{\boldsymbol \omega}^{p-2}
	\Big( \frac{\boldsymbol \omega}{4 \boldsymbol \mu^-} \Big)^{q+1-p}.
$$
By \eqref{Eq:6:9} and  $q+1>p$, the preceding inequality is satisfied if
$$
	\sigma^p
	\leq
	  \widehat{\boldsymbol \omega}^{p-2}
	\Big( \frac{1-\eta}{4 (1+\xi)} \Big)^{q+1-p}.
$$
Thus, we have chosen $\sigma$ and $\boldsymbol{\widehat{\om}} $ to verify \eqref{Eq:5:8}.
Consequently, Proposition \ref{Prop:5:1} gives that
$$
	\essosc_{Q_r(\widehat{\boldsymbol \omega}^{2-p})} \widehat{w}
	\leq
	\boldsymbol \gamma \widehat{\boldsymbol \omega} \Big( \frac{r}{\rho} \Big)^{\beta_1}
$$
for any $0 < r < \rho$, with the indicated constants $\boldsymbol \gamma > 0$ and $\beta_1 \in (0,1)$.
By definition of $\widehat{w}$, this means that
$$
	\essosc_{Q_r(\widehat{\boldsymbol \omega}^{2-p} (\boldsymbol \mu^-)^{q+1-p})} w
	\leq
	\boldsymbol \gamma \Big( \frac{r}{ \rho} \Big)^{\beta_1}
$$
for any $0 < r <  \rho$.
Setting $c := \widehat{\boldsymbol \omega}^{2-p} \xi^{q+1-p}$ and using \eqref{Eq:6:9}, we have that 
$$
Q_r(c \boldsymbol \omega^{q+1-p}) \subset Q_r(\widehat{\boldsymbol \omega}^{2-p} (\boldsymbol \mu^-)^{q+1-p}).
$$
Moreover, the mean value theorem, \eqref{Eq:5:8} and the preceding estimate yield
\begin{align*}
	\essosc_{Q_r(c \boldsymbol \omega^{q+1-p})} v
	&\leq
	\tfrac1q
	\max \Big\{ \big|\essup_{Q_r(c \boldsymbol \omega^{q+1-p})} w\big|^\frac{1-q}{q}, \big|\essinf_{Q_r(c \boldsymbol \omega^{q+1-p})} w \big|^\frac{1-q}{q} \Big\}
	\essosc_{Q_r(c \boldsymbol \omega^{q+1-p})} w \\
	&\leq
	\boldsymbol\gamma
	\max \Big\{ 1, \Big( \frac{1+\xi}{\xi} \Big)^{1-q} \Big\}
	\Big( \frac{r}{ \rho} \Big)^{\beta_1},
\end{align*}
for any $0 < r <  \rho$.
Reverting to $u$ and the index $j$ and using \eqref{Eq:6:9}, this shows that
\begin{equation}\label{Eq:case2}
	\essosc_{Q_r(c \boldsymbol{\omega}_j^{q+1-p})} u
	\leq
	\boldsymbol{\gamma} \boldsymbol{\mu^-}_j \Big( \frac{r}{ \rho_j} \Big)^{\beta_1}
	\leq
	\boldsymbol{\gamma} \boldsymbol{\omega}_j \Big( \frac{r}{ \rho_j} \Big)^{\beta_1}
\end{equation}
for all $0 < r <  \rho_j$.
%%%%%%%%%%%%%%%%%%%%

\subsection{The final argument}
Our next goal is to combine the reduction of oscillation \eqref{Eq:case1} and \eqref{Eq:case2}, and to remove the qualitative knowledge on the index $j$ in the final oscillation estimate.
To this end, we first observe that according to \eqref{Eq:case2},
\begin{equation}
	\essosc_{Q_{\widetilde{\rho}_i}(c \boldsymbol{\omega}_j^{q+1-p})} u
	\leq
	\boldsymbol{\gamma} (1-\eta)^i \boldsymbol{\omega}_j
	\equiv
	\boldsymbol{\gamma} \boldsymbol{\omega}_{j+i}
	\label{eq:start conclusion}
\end{equation}
for $\widetilde{\rho}_i := (1-\eta)^\frac{i}{\beta_1} \rho_j$, $i \in \nn$.
Further, for $i \in \nn \cup \{0\}$ we set
$$
\widehat{\rho}_i := \min\big\{ \lambda, (1-\eta)^\frac{1}{\beta_1} \big\}^i \rho \quad \text{ and } \quad \widehat{c} := \min\big\{  4^{p-q-1}, c \big\}.
$$
As a result, one may verify the set inclusions: 
\begin{equation*}%\label{Eq:para}
	\left\{
	\begin{aligned}
		&Q_{\widehat{\rho}_i}(\widehat{c} \boldsymbol{\omega}_i^{q+1-p}) \subset Q_i  &\text{ for } i = 1, \cdots, j, \\ 
		&Q_{\widehat{\rho}_i}(\widehat{c} \boldsymbol{\omega}_i^{q+1-p}) \subset Q_{\widetilde{\rho}_{i-j}}(c \boldsymbol{\omega}_j^{q+1-p})  &\text{ for } i = j+1, \cdots
	\end{aligned}
	\right.
\end{equation*}
This allows us to combine \eqref{Eq:case1} and \eqref{eq:start conclusion} and  to obtain that
\begin{equation*}
	\essosc_{Q_{\widehat{\rho}_i}(\widehat{c} \boldsymbol{\omega}_i^{q+1-p})} u
	\leq
	\boldsymbol{\gamma} \boldsymbol{\omega}_i
	\qquad \forall \, i \in \nn.
	%\label{eq:uniform decay}
\end{equation*}
Now, consider $r \in (0, \rho)$.
Observe that there must exist $i \in \nn \cup \{0\}$ satisfying
\begin{equation}
	\boldsymbol{\omega}_{i+1}^{q+1-p} \widehat{\rho}_{i+1}^p
	<
	\boldsymbol{\omega}^{q+1-p} r^p
	\leq
	\boldsymbol{\omega}_i^{q+1-p} \widehat{\rho}_i^p.
	\label{eq:r squeezed}
\end{equation}
By \eqref{eq:r squeezed}$_2$ (which in particular implies $r \leq \widehat{\rho}_i$),
we find that
\begin{equation}
	\essosc_{Q_r(\widehat{c} \boldsymbol{\omega}^{q+1-p})} u
	\leq
	\essosc_{Q_{\widehat{\rho}_i}(\widehat{c} \boldsymbol{\omega}_i^{q+1-p})} u
	\leq
	\boldsymbol{\gamma} \boldsymbol{\omega}_i
	=
	\boldsymbol{\gamma} (1-\eta)^i \boldsymbol{\omega}.
	\label{eq:independent cylinder}
\end{equation}
Moreover, by \eqref{eq:r squeezed}$_1$ we obtain that
\begin{align*}
	\boldsymbol{\omega}^{q+1-p} r^p
	>
	\boldsymbol{\omega}_{i+1}^{q+1-p} \widehat{\rho}_{i+1}^p
	&=
	\Big( (1-\eta)^{q+1-p} \min\Big\{ \lambda, (1-\eta)^\frac{1}{\beta_1} \Big\} \Big)^{i+1}
	\boldsymbol{\omega}^{q+1-p} \rho^p\\
	&=:\bar{c}^{i+1}\boldsymbol{\omega}^{q+1-p} \rho^p,
\end{align*}
from which we deduce that
\[
\Big( \frac{r}{ \rho} \Big)^p>\bar{c}^{i+1}.
\]
Thus, joining this inequality with the right-hand side of \eqref{eq:independent cylinder},
we conclude that
$$
	\essosc_{Q_r(\widehat{c} \boldsymbol{\omega}^{q+1-p})} u
	\leq
	%\boldsymbol{\gamma} (1-\eta)^{-1} \boldsymbol{\omega} \left( \frac{r}{A \rho} \right)^{\beta_2}
	%=
	\boldsymbol{\gamma} \boldsymbol{\omega} \Big( \frac{r}{ \rho} \Big)^{\beta},\quad\text{ with }\be:=\frac{p\ln(1-\eta)}{\ln\bar{c}},
$$
for any $0 < r <  \rho$.
Finally, we complete the proof of Theorem \ref{Thm:1:1} by means of a standard covering argument.

The assumption that $\boldsymbol{\omega}\le1$ is not restrictive. For otherwise, one could work with
the rescaled function
\[
v(x,t):=(2M)^{-1} u\big( x,(2M)^{q+1-p}t \big)\quad\text{ where } M=\|u\|_{\infty, E_T},
\]
which verifies the same type of equations \eqref{Eq:1:1} -- \eqref{Eq:1:2p}.

%%%%%%%%%%%%%%%%%%%%%%%%%%%%%%%%%%%%%%%%

\end{document}